\documentclass[reqno,12pt]{amsart}
\usepackage{graphicx}
\usepackage{epsfig}
\usepackage{amssymb,amsthm,amsmath,amsfonts,amstext,mathrsfs,fancybox}
\usepackage{enumerate}
\usepackage{latexsym}
\usepackage{epsfig}
\usepackage{url}
\usepackage{fancyhdr}
\usepackage{color}
\usepackage[unicode]{hyperref}

\hypersetup{colorlinks=true,
linkcolor=blue,
citecolor=blue,
urlcolor=blue,
filecolor=blue,
bookmarksnumbered=true,
pdfstartview=FitH,
pdfhighlight=/N}

\newtheorem*{thmm}{Theorem}

\newtheorem{prop}{Proposition}
\newtheorem{prob}{Problem}

\input xy
\xyoption{all}

\def\dollar{\$}

\title[The modular group and words]
{The modular group and words in its two generators}
\author[G. Alkauskas]{Giedrius Alkauskas}
\address{Vilnius University, Department of Mathematics and Informatics, Naugarduko 24, LT-03225 Vilnius, Lithuania}
\email{giedrius.alkauskas@mif.vu.lt}

\newcounter{noteno}

\newenvironment{Example}
	{\refstepcounter{noteno}
	\begin{small}
	\medbreak\par\noindent{{\bf Example~\thenoteno}.\,}}
	{\hfill{$\Box$}\end{small}\par\medbreak}

\begin{document}
\begin{abstract}
Consider the full modular group $\sf{PSL}_{2}(\mathbb{Z})$ with presentation $\langle U,S|U^3,S^2\rangle$. Motivated by our investigations on quasi-modular forms and the Minkowski question mark function (so that this paper might be considered as a necessary appendix), we are lead to the following natural question. Some words in the alphabet $\{U,S\}$ are equal to the unity; for example, $USU^3SU^2$ is such a word of length $8$, and $USU^3SUSU^3S^3U$ is such a word of length $15$. We consider the following integer sequence. For each $n\in\mathbb{N}_{0}$, let $t(n)$ be the number of words in alphabet $\{U,S\}$ that equal the identity in the group. This is the new entry A265434 into the Online Encyclopedia of Integer Sequences. We investigate the generating function of this sequence and prove that it is an algebraic function over $\mathbb{Q}(x)$ of degree $3$. As an interesting generalization, we formulate the problem of describing all algebraic functions with a Fermat property. 
\end{abstract}
\date{\today}
\thanks{The research of the author was supported by the Research Council of Lithuania grant No. MIP-072/2015.}

\maketitle

\section{Introduction}
\subsection{Motivation and results}
The topic of this paper is the following two new sequences. The first one $t(n)$, $n\geq 0$, starts from
\begin{eqnarray}
1, 0, 1, 1, 1, 5,  2, 14, 13, 31, 66, 77, 240, 286, 722, 1226, 2141, 4760, 7268, 16473,\ldots
\label{seka}
\end{eqnarray}  
This is the sequence A265434 in \cite{oeis}. The second sequence $\mathfrak{t}(n)$, $n\geq 0$, starts from
\begin{eqnarray*}  
1,0,1, 1, 0, 3, 0, 5, 3, 7, 16, 12, 50, 44, 123, 195, 301, 718, 928, 2244,\ldots
  \end{eqnarray*}
  They are defined as follows. Let $U^3=I$ and $S^{2}=I$ be two elements of order $3$ and $2$, respectively, $I$ being the unity. Consider the \emph{the full modular group} ${\sf PSL}_{2}(\mathbb{Z})$. It is known that it is freely generated by $U=\left(\begin{array}{cc}0 & 1 \\-1 & 1 \\ \end{array}\right)$, element of order $3$, and  and $S=\left(\begin{array}{cc}0 & 1 \\-1 & 0 \\ \end{array}\right)$, element of order $2$.  Let $n\in\mathbb{N}_{0}$, and $A$ be any word in the alphabet $U,S$ of total lenght $n$: 
\begin{eqnarray*}
A=\prod\limits_{j=1}^{n}(U^{\epsilon_{j}}S^{\delta_{j}}),\quad\epsilon_{j},\delta_{j}\in\{0,1\},\quad
\epsilon_{j}+\delta_{j}=1.
\end{eqnarray*}
Let $t(n)$ be the number of such words that in the group $\Gamma$ are equal to the unity. Our method to calculate the first $20$, $0\leq n\leq 19$, terms of the sequence (\ref{seka}) was via a brute force calculation. As the model, let $U$ and $S$ be given my $2\times 2$ matrices as above. Let us construct a binary tree, starting from the node $I=\left(\begin{array}{cc}1 & 0 \\0 & 1 \end{array}\right)$. Each node $A$ in this binary tree generates two offspring - the right one $AU$, and the left one $AS$. For a given $n$, we then calculate which of the $2^{n}$ matrices in the $n$th generation are equal to $\pm I$. Via this method, a standard home computer can give few more terms of this sequence. \\  

Further, we call such a word \emph{primitive}, if
\begin{eqnarray*}
\prod\limits_{j=1}^{s}(U^{\epsilon_{j}}S^{\delta_{j}})=I\text{ only if }s=n.
\end{eqnarray*} 
Let $\mathfrak{t}(n)$ be the number of primitive (and nonempty) words of length $n$. Let us introduce the generating functions
\begin{eqnarray*}
\sum\limits_{n=0}^{\infty}t(n)x^{n}=T(x),\quad
\sum\limits_{n=1}^{\infty}\mathfrak{t}(n)x^{n}=\mathfrak{T}(x).
\end{eqnarray*}
Obviously,
\begin{eqnarray*}
\frac{1}{1-\mathfrak{T}(x)}=T(x).
\end{eqnarray*}
Indeed, we are just considering words broken up into primitive loops, so 
\begin{eqnarray*}
T(x)=1+\big{(}\mathfrak{t}(1)x+\mathfrak{t}(2)x^2+\cdots\big{)}+
\big{(}\mathfrak{t}(1)x+\mathfrak{t}(2)x^2+\cdots\big{)}^2+\cdots.
\end{eqnarray*} 
The number of total words of length $n$ is equal to $2^{n}$. So, $t(n)\leq 2^{n}$, and thus both series converge at least for $|x|<\frac{1}{2}$. The Table gives values for this function (and all words) for $0\leq n\leq 8$.

\noindent\begin{center}
\begin{tabular}{|r | r | r|| l |}
\hline
\multicolumn{4}{|c|}{\textbf{Table. }Sequences $t(n)$ and $\mathfrak{t}(n)$. Primitive elements are in bold}\\
\hline
$n$ & $t(n)$& $\mathfrak{t}(n)$ & Products \\
\hline
$0$ & $1$& $0$& $I$        \\
$1$ & $0$& $0$& $-$      \\
$2$ & $1$& $1$& $\mathbf{S^2}$      \\
$3$ & $1$& $1$& $\mathbf{U^3}$        \\
$4$ & $1$& $0$& $S^4$        \\
$5$ & $5$& $3$& $U^3S^2$, $\mathbf{U^2S^2U}$, $\mathbf{US^2U^2}$, $S^2U^3$, 
$\mathbf{SU^3S}$    \\
$6$ & $2$&$0$& $S^6$, $U^6$        \\
$7$ & $14$& $5$& $U^3S^4$, $\mathbf{U^2S^4U}$, $\mathbf{US^4U^2}$, $S^4U^3$, $SU^3S^3$, $S^2U^2S^2U$, $\mathbf{SU^2S^2US}$, $U^2S^2US^2$  \\
$ $ & $ $& $ $& $S^2US^2U^2$, $\mathbf{SUS^2U^2S}$, $US^2U^2S^2$, 
$S^3U^3S$, $S^2U^3S^2$, $\mathbf{US^2US^2U}$ \\
$8$ & $13$& $3$& $S^8$, $U^6S^2$, $U^5S^2U$, $U^4S^2U^2$, $U^3S^2U^3$, $U^2S^2U^4$, $US^2U^5$, $S^2U^6$  \\
$ $ & $ $& $ $& $\mathbf{SU^6S}$, $U^3SU^3S$, $\mathbf{U^2SU^3SU}$, $\mathbf{USU^3SU^2}$, $SU^3SU^3$ \\
\hline
\end{tabular}
\end{center}

The integer $q(n,m)$ counts the number of words that are equal to the unity in the group $\Gamma$, and which contain $n$ copies of $U$ and $m$ copies of $S$. Note that
\begin{eqnarray*}
q(n,m)=0\text{ unless }n\equiv 0\text{ (mod } 3)\text{ and }m\equiv 0\text{ (mod } 2). 
\end{eqnarray*}

We also introduce
\begin{eqnarray*}
\sum\limits_{A\text{ is unity}}x^{\sum\epsilon_{j}}y^{\sum\delta_{j}}=
\sum\limits_{n,m=0}^{\infty}q(n,m)x^{n}y^{m}=Q(x,y).
\end{eqnarray*}
Obviously, $T(x)=Q(x,x)$. The series $T(x)$ can be interpreted as a return generating function for a certain directed graph (see the Subsection \ref{sub1.1} and Figure \ref{fig1}), and thus this function belongs to a hugely diverse and abundant family of functions with the following two features. All of them share the property that their Taylor coefficients grow exponentially, and all are related to the graph theory and enumeration (\cite{finch}, 5.6).\\

The main result of this paper is the following 
\begin{thmm}
\label{thm1}
The function $Q(x,y)$ is an algebraic $3$rd degree function over $\mathbb{Q}(x,y)$, satisfying \footnotesize
\begin{eqnarray*}
(y^6-x^6+6y^2x^3-3y^4+2x^3+3y^2-1)Q^3+(x^3y^2-y^4+x^3+2y^2-1)Q^2+(x^3-y^2+1)Q+1=0.
\end{eqnarray*}
\normalsize
In particular, the function $T(x)$ is an algebraic $3$rd degree function over $\mathbb{Q}(x)$, satisfying
\begin{eqnarray*}
(6x^5-3x^4+2x^3+3x^2-1)T^3+(x^5-x^4+x^3+2x^2-1)T^2+(x^3-x^2+1)T+1=0.
\end{eqnarray*}
\end{thmm}
Thus,
\begin{eqnarray*}
\sum\limits_{n=0}^{\infty}\frac{t(n)}{2^{n}}=\frac{14}{13}+\frac{6}{13}\sqrt{17},\quad
\sum\limits_{n=0}^{\infty}\frac{t(n)}{2^{n}}\cdot\frac{1}{2^{n+1}}=0.5443390725_{+}.
\end{eqnarray*}
The last number can be interpreted as probability that a randomly chosen word is a unity, with a convention that each length $n$ is given a weight $2^{-(n+1)}$, and then probability is equally distributed among all words of length $n$.\\

Two proofs of this result are given. The first one is longer, but uses only elementary combinatorics and considerations from the scratch. The second one is shorter, is included due to a suggestion and very clear guidance by the referee, and is a standard proof in the area of geometric and combinatoric group theory.\\ 

We finish this Subsection with the following 
\begin{prop}For any prime $p>3$, $t(p)\equiv 0\text{ (mod }p)$.
\label{prop1}
\end{prop}
\begin{proof}Indeed, if $AB=I$ in the group $\Gamma$, then $BA=I$ as well. So, any cyclic permutation of the word which is a unity is a unity again. So, words of length $p>3$ which are equal to the unity split into groups each containing exactly $p$ words. Indeed, otherwise all these permutations are equal. But $U^{p}\neq I$ and $S^{p}\neq I$ - a contradiction.
\end{proof}
\subsection{Context, previous results}
\label{sub1.1}
In \cite{alkauskas} we investigate the relation of modular forms to the Minkowski question mark function, introducing the notion of \emph{mean-modular forms}. In particular, the analytic continuation of a certain bivariate analytic function $G(\varkappa,z)$ (an extension of the Stieltjes transform of the Minkowski question mark function) requires us to know the analytic formula for $Q(x,y)$ as a bivariate function; whence the principal motivation for the current paper. In fact, the result of Theorem is new, but it is an exercise for people in the field. The much more ambitious program of integrating the world of Minkowski question mark function and the world of modular forms for $\sf{PSL}_{2}(\mathbb{Z})$ justifies the current paper as a necessary appendix to \cite{alkauskas}.\\
 
As just mentioned, this particular question is new, though many intricately related problems were investigated and solved before, and the topic itself is of big importance in the theory of groups (growth and cogrowth rates), graphs (return and first-return paths), and non-commutative probability.\\

As a particular example of his more general results, Kuksov \cite{kuksov} considers the cogrowth rate of the product $\mathbb{Z}/(2)\star\mathbb{Z}/(3)$, which is the same modular group. The question of investigating cogrowth rates amounts to the following. Count the number of \emph{reduced} words in the alphabet $\{U,U^{-1},S,S^{-1}\}$ that are equal to the unity. \emph{Reduced} means that $S$ and $S^{-1}$ , and also $U$ and $U^{-1}$ never follow immediately one after another. It is obvious then that the total number of reduced words of length $n$ is $4\cdot3^{n-1}$: the first letter can be anything, after that the choice is restricted. The generating function of this sequence (cogrowth series) turns out to be
\begin{eqnarray*}
v(x)=\frac{(x+1)\Big{(}9x^5-3x^4+8x^3-x^2+x-(6x^2-x+2)\sqrt{R(x)}\Big{)}}{2(3x-1)(3x^2+1)(3x^2+3x+1)(3x^2-x+1)},\\
\text{where }R(x)=81x^8-54x^7+9x^6-18x^5-8x^4-6x^3+x^2-2x+1.
\end{eqnarray*}
The cogrowth rate (the inverse of the radius of convergence of the Taylor series for this function at the origin) turns out to be $2.9249_{+}<3$. The function $v(x)$ is quadratic algebraic function, as opposed to cubic algebraic function $T(x)$ in our case. The Taylor coefficients of $v(x)$ are 
\begin{eqnarray*}
1, 0, 2, 2, 6, 24, 44, 136, 298, 914, 2462, 6464,\ldots
\end{eqnarray*} 
For example, there are $6$ reduced words of length $4$:
\begin{eqnarray*}
SSSS, S^{-1}S^{-1}S^{-1}S^{-1}, U^{-1}SSU, USSU^{-1}, U^{-1}S^{-1}S^{-1}U, US^{-1}S^{-1}U^{-1}.
\end{eqnarray*}
Also, Proposition \ref{prop1} does not have an analogoue in this case. \\

In a related direction, Quenell in \cite{quenell} investigates the \emph{return generating function} of Cayley graphs for the free products of finite groups. This is related to the thesis of McLaughlin \cite{thesis} (see also \cite{bartholdi}). In particular, in case of the modular group the return generating function counts words in $\{U, U^{-1}, S, S^{-1}\}$ of total length $n$ that are equal to the unity (words are not necessarily reduced). The total number of all words of spell length $n$ is $4^{n}$. In this context, the function $T(x)$ can be interpreted as a return generated function for a \emph{directed} Cayley graph for $\mathbb{Z}/(2)\star\mathbb{Z}/(3)$, where each $2$-cycle and $3$-cycle has a particular direction chosen in advance; see Figure \ref{fig1}.\\

Franz Lehner has pointed out that the question in consideration is a special case of a ``free convolution" and can be obtained via Voiculescu-Woess transform \cite{lehner,voiculescu, woess1, woess2}. This technique is implemented as a package and it is part of the library of FriCAS. So, the current paper can be thought as a purely combinatoric demonstration of the result, with emphasis on a bivariate function $Q(x,y)$ rather than a univariate $T(x)=Q(x,y)$.   

\begin{figure}
\epsfig{file=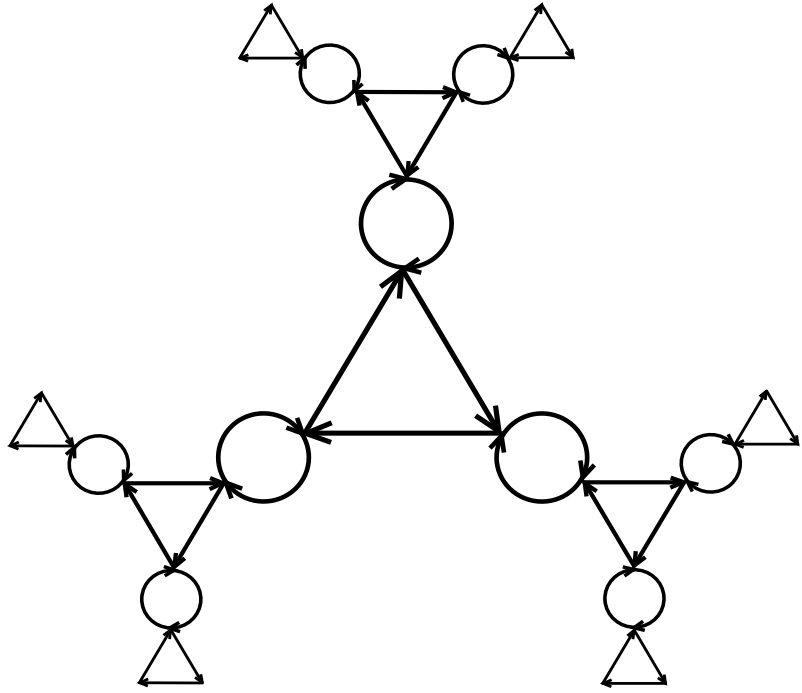,width=300pt,height=260pt,angle=0}
\caption{First few generations of a directed Cayley graph for $\mathbb{Z}/(2)\star\mathbb{Z}/(3)$. }
\label{fig1}
\end{figure}

\subsection{Algebraic functions with a Fermat property}
Before passing to the proof of the main result, we will formulate one interesting problem, which is motivated by Proposition \ref{prop1}.\\

Consider the following rational function
\begin{eqnarray*}
P(x)=\frac{x^2}{(1-2x)(1-x)}=\sum\limits_{n=2}^{\infty}s(n)x^{n}=
\sum\limits_{n=2}^{\infty}(2^{n-1}-1)x^{n}.
\end{eqnarray*}
Thus, we have $s(p)\equiv 0\text{ (mod }p)$ for $p>2$ prime.\\

Next, let
\begin{eqnarray*}
J(x)=\frac{1}{\sqrt{1-4x}}-\frac{2}{1-x}=\sum\limits_{n=0}^{\infty}s(n)x^{n}=\sum\limits_{n=0}^{\infty}\Big{(}\binom{2n}{n}-2\Big{)}x^{n}.
\end{eqnarray*}
This also gives $s(p)\equiv 0\text{ (mod }p)$ for $p\geq2$ prime.\\

Finally, as is implied by Proposition \ref{prop1}, we have the same (for $p>3$) conclusion for a degree $3$ algebraic function $T(x)=1+x^2+O(x^3)$, which satisfies
\begin{eqnarray*}
(6x^5-3x^4+2x^3+3x^2-1)T^3+(x^5-x^4+x^3+2x^2-1)T^2+(x^3-x^2+1)T+1=0.
\end{eqnarray*}
Let $R(x)=\sum\limits_{n=0}^{\infty}s(n)x^{n}\in\mathbb{Z}[[x]]$ be an algebraic function over $\mathbb{Q}(x)$, unramified at $x=0$.  Suppose, $s(p)\equiv 0\text{ (mod }p)$ for all sufficiently large prime numbers $p$. We call such a function $R(x)$ \emph{an algebraic function with a Fermat property}. Thus, we formulate
\begin{prob}
\label{prob1}
 Characterize all algebraic functions with a Fermat property in general, and in any particular algebraic function field $\mathbb{Q}(x,U)$, where $U$ is unramified at $x=0$.
\end{prob}
We can multiply the function $R(x)$ by an integer to get the congruence valid for all primes. All algebraic functions with a Fermat property form an abelian group $\mathscr{F}$. Since the set of all algebraic functions over $\mathbb{Q}$ is countable, $\mathscr{F}$ is also countable. If $U(x)=\sum_{n=0}^{\infty}a(n)x^{n}\in\mathbb{Z}[[x]]$ is an algebraic function, then $U^{\partial}(x):=xU'(x)=\sum_{n=0}^{\infty}na(n)x^{n}\in\mathscr{F}$. Indeed, first it is obvious that $U^{\partial}\in\mathbb{Z}[[x]]$. And second, if $G(Y,x)\in\mathbb{Z}[Y,x]$ and $G(U,x)=0$, then
\begin{eqnarray*}
U'=-\frac{G_{x}(U,x)}{G_{Y}(U,x)}
\end{eqnarray*} 
belongs to the same algebraic function field $\mathbb{Q}(x,U)$. Let $D$ (from ``Differential") be the union of all such possible $U^{\partial}$. Then $D$ is a subgroup of $\mathscr{F}$. So is $\mathbb{Z}[x]$. Finally, let $U(x)=\sum_{n=0}^{\infty}a(n)x^{n}\in\mathbb{Z}[[x]]$ is again an algebraic function, unramified and without a pole at $x=0$, and let for an integer $M\geq 2$, $U^{(M)}(x)=U(x^{M})$. Let $P$ (from ``Power") be the group whose elements are 
\begin{eqnarray*}
\sum\limits_{j=1}^{s}U_{j}^{(M_{j})},\quad M_{j}\geq 2.
\end{eqnarray*}
Then $P$ is also a subgroup of $\mathscr{F}$. Of course, any two of the subgroups $D$, $\mathbb{Z}[x]$ and $P$ have a non-trivial pairwise intersection. Let also for any $U$, algebraic over $\mathbb{Q}(x)$ and unramified at $x=0$, $\mathscr{F}_{U}=\mathscr{F}\cap\mathbb{Q}(x,U)$, $D_{U}=D\cap\mathbb{Q}(x,U)$, $P_{U}=P\cap\mathbb{Q}(x,U)$ (We identify any algebraic function with its Laurent expansion at $x=0$). We may refine Problem \ref{prob1} as follows.
\begin{prob}
\label{propb2}
Find the structure of abelian groups
\begin{eqnarray*}
\mathscr{F}/(D+\mathbb{Z}[x]+P),\quad \mathfrak{A}_{U}=\mathscr{F}_{U}/\big{(}D_{U}+\mathbb{Z}[x]\big{)},\quad
\mathfrak{P}_{U}=\mathscr{F}_{U}/\big{(}D_{U}+\mathbb{Z}[x]+P_{U}\big{)}
\end{eqnarray*}
for any $U$ algebraic over $\mathbb{Q}(x)$ and unramified at $x=0$. In particular, for example, what is the group $\mathfrak{P}_{\varnothing}$ (that is, we talk only about $\mathbb{Q}(x)$)?
\end{prob}
In fact, for a function $J(x)$ we have an even stronger property $s(p)\equiv 0\text{ (mod }p^{2})$. We may call it \emph{an algebraic function with a strong Fermat (or Wieferich) property}, and ask similar questions.
\section{The first proof}
Let us introduce the function $\hat{q}(n,m)$. This is defined similarly as $q(n,m)$. Namely, we count the number of words in $U,S$ which are equal to unity in $\Gamma$, which have $n$ copies of $U$ and $m$ copies of $S$, but which  do not contain $S^{2}$ (two $S's$ in a row). Let
\begin{eqnarray*}
\sum\limits_{n,m\geq 0}^{\infty}\hat{q}(3n,2m)x^{3n}y^{2m}=\widehat{Q}(x,y).
\end{eqnarray*}
\begin{prop}
\label{prop2}
We have an identity \begin{eqnarray*}
q(3n,2m)=\sum\limits_{k\geq 0}\hat{q}(3n,2m-2k)\cdot\binom{3n+k}{k}.
\end{eqnarray*}
This implies
\begin{eqnarray*}
Q(x,y)=\widehat{Q}\Big{(}\frac{x}{1-y^2},y\Big{)}\cdot\frac{1}{1-y^2}.
\end{eqnarray*}
\end{prop}
\begin{proof}
Indeed, consider any word which has $3n$ copies of $U$ and $2m$ copies of $S$. Now, replace each occurring segment $S^{2\ell+j}$, $\ell\geq 0$, $j\in\{0,1\}$, with $S^{j}$. We get a word which lies in the set which defines $\hat{q}(3n,2m-2k)$ for some $k\in\mathbb{N}_{0}$. In the other direction, consider the latter word. Suppose, we have $k$ spare copies of $S^2$. We can plug $k$ copies of $S^{2}$ into such a word, to get a word which defines $q(3n,2m)$. As can be seen, we can confine in plugging to the left of each occurrence of $U$, plus to the right of the rightmost $U$. This gives $3n+1$ possible places to plug in. We want to distribute $k$ copies of $S^2$. This gives the formula in Proposition.      
\end{proof}

Let us divide the words which define the quantity $\hat{q}(n,m)$ into $7$ disjoint subsets. $B$ stands for any non-empty word. Let $n,m\geq 0$. Here everywhere ``---" stands for the phrase ``be the number of such words that are of the form...".
\begin{itemize}
\item[a)] $a(n,m)$ --- $SUBUS$.
\item[b)] $b(n,m)$ --- either $U^{\alpha}SBSU^{\beta}$, $\alpha,\beta>0$ and $\alpha+\beta\equiv 0\text{ (mod } 3)$, or $U^{\gamma}$, $\gamma\equiv 0\text{ (mod } 3)$.
\item[c)] $c(n,m)$ --- $U^{\alpha}SBS$, or $SBSU^{\alpha}$, $\alpha>0$, $\alpha\equiv 0\text{ (mod } 3)$.
\item[d)] $d(n,m)$ --- $U^{\alpha}SBSU^{\beta}$, $\alpha,\beta>0$, and $\alpha+\beta\equiv 2\text{ (mod } 3)$.
\item[e)] $e(n,m)$ --- $U^{\alpha}SBSU^{\beta}$, $\alpha,\beta>0$, and $\alpha+\beta\equiv 1\text{ (mod }3)$.
\item[f)] $f(n,m)$ --- either $SBSU^{\alpha}$, or $U^{\alpha}SBS$, $\alpha\equiv1\text{ (mod }3)$.
\item[g)] $g(n,m)$ --- either $SBSU^{\alpha}$, or $U^{\alpha}SBS$, $\alpha\equiv2\text{ (mod }3)$.
\end{itemize}
Let us also introduce two subsets which define $d$ and $f$, respectively, as follows:
\begin{itemize}
\item[$\mathfrak{d}$)]$\mathfrak{d}(n,m)$ --- $USBSU$,
\item[$\mathfrak{f}$)]$\mathfrak{f}(n,m)$ --- $SBSU$.
\end{itemize}
\begin{Example}We have: $\hat{q}(6,2)=5$. Words which corespond to $a,b,c$ are, respectively, $\{SU^6S\}$, $\{U^2SU^3SU,USU^3SU^2\}$, $\{SU^3SU^3,U^3SU^3S\}$. All sets beyond ``c" are empty. 
\end{Example} 
\begin{Example}We have: $\hat{q}(9,4)=20$. Words which correspond to $a(9,4)=7$, $b(9,4)=5$, $c(9,4)=3$, $d(9,4)=1$, $e(9,4)=0$, $f(9,4)=2$, $g(9,4)=2$,  are, respectively:\\
\noindent $a:\{SUSU^3SU^5S, SUSU^6SU^2S, SU^2SU^3SU^4S, SU^2SU^6SUS, SU^3SU^3SU^3S, SU^4SU^3SU^2S,\\
SU^5SU^3SUS\}$,\\
$b:\{USUSU^3SU^2SU^2, USU^2SU^3SUSU^2, U^2SUSU^3SU^2SU, U^2SU^2SU^3SUSU, SUSU^3SU^2SU^3\}$,\\
$c:\{SU^2SU^3SUSU^3,U^3SUSU^3SU^2S, U^3SU^2SU^3SUS\}$,\\
$d:\{USU^3SUSU^3SU\}$, \\
$e:\varnothing$,\\ 
$f:\{SU^3SU^2SU^3SU, USU^3SU^2SU^3S\}$,\\
$g:\{SU^3SUSU^3SU^2, U^2SU^3SUSU^3S\}$. 
\label{ex2}
\end{Example}

Note that $a(n,m)$ and other $8$ functions are potentially non-zero only for $n=3k$, $m=2l$ for $k\geq 0$, $l\geq 0$.\\ 

The first step to derive our main result is the following 
\begin{prop}
We have the following recurrences:
\begin{eqnarray*}
a(3n,2m)&=&b(3n,2m-2)+d(3n,2m-2)+e(3n,2m-2),\\
b(3n,2m)&=&\sum\limits_{k\geq 1}(3k-1)a(3n-3k,2m)\text{ for }m\geq 1,\quad b(3n,0)=1\text{ for }n>0,\\
c(3n,2m)&=&2\sum\limits_{k\geq 1}a(3n-3k,2m),\\
d(3n,2m)&=&\sum\limits_{k\geq 0}(3k+1)\mathfrak{d}(n-3k,2m),\\
e(3n,2m)&=&\sum\limits_{k\geq 1}3k\mathfrak{f}(n-3k,2m),\\
f(3n,2m)&=&2\sum\limits_{k\geq 0}\mathfrak{f}(n-3k,2m),\\
g(3n,2m)&=&2\sum\limits_{k\geq 0}\mathfrak{d}(n-3k,2m).
\end{eqnarray*}
These hold for $n,m\geq 0$ assuming that all these functions vanish if one of the arguments is negative.
\label{prop3}
\end{prop}
\begin{proof}All these equalities are straighforward. 
Only the formula for $e$ needs an explanation. Indeed, we note that (as already used in the proof of Proposition \ref{prop1}) if $AB=I$ in the group $\Gamma$, then $BA=I$. So, if $U^2BU^2=I$, this gives $BU=I$. 
\end{proof}
Let 
\begin{eqnarray*}
A(x,y)=\sum\limits_{n,m\geq 0}a(3n,2m)x^{3n}y^{2m}
\end{eqnarray*}
be the generating function of the coefficients $a(3n,2m)$, and similarly we define other bivariate functions $B,C,D,E,F,G,\mathfrak{D}$ and $\mathfrak{F}$. The identities of Proposition \ref{prop3} now read as
\begin{eqnarray*}
A=y^2B+y^2D+y^2E,&\quad&
B=\frac{x^3}{1-x^3}+\frac{x^3(x^3+2)}{(1-x^3)^2}A,\\
C=\frac{2x^3}{1-x^3}A,\quad
D=\frac{2x^3+1}{(1-x^3)^2}\mathfrak{D},&&
E=\frac{3x^3}{(1-x^3)^2}\mathfrak{F},\quad
F=\frac{2}{1-x^3}\mathfrak{F},\quad
G=\frac{2}{1-x^3}\mathfrak{D}.
\end{eqnarray*}
This gives
\begin{eqnarray}
A&=&\frac{y^2x^3}{1-x^3}+\frac{y^2x^3(x^3+2)}{(1-x^3)^2}A+  \frac{y^2(2x^3+1)}{(1-x^3)^2}\mathfrak{D}+\frac{3y^2x^3}{(1-x^3)^2}\mathfrak{F}.
\label{a-df}
\end{eqnarray}
Our function $\widehat{Q}$ is then
\begin{eqnarray*}
\widehat{Q}=1+A+B+C+D+E+G.
\end{eqnarray*}
That is,
\begin{eqnarray}
\widehat{Q}(x,y)=\frac{1}{1-x^3}+\frac{2x^3+1}{(1-x^3)^2}A
+\frac{3}{(1-x^3)^2}\mathfrak{D}+\frac{x^3+2}{(1-x^3)^2}\mathfrak{F}.
\label{tarpp}
\end{eqnarray}

\begin{Example}The sets which define $``d"$ through $``g"$ are non-empty starting only from the length $13$ (that is, where $3n+2m=13$). Thus, if we use the above recurences only minding the sets $``a"$, $``b"$ and $``c"$ (that is, assuming that $\mathfrak{D}=0$ and $\mathfrak{F}=0$), we obtain
\begin{eqnarray*}
A(x,y)\gg \frac{y^2x^3(1-x^3)}{(1-x^3)^2-y^2x^3(x^3+2)}.
\end{eqnarray*}
By the sign $\gg$ we mean that the inequality $\geq$ holds for Taylor coeffiencts of corresponding functions on the left and on the right.
This gives
\begin{eqnarray*}
\widehat{Q}(x,y)\gg 1+A+B+C\gg\frac{1-x^3-y^2x^3}
{(1-x^3)^2-y^2x^3(x^3+2)},
\end{eqnarray*}
and consequently
\begin{eqnarray*}
T(x)=\widehat{Q}\Big{(}\frac{x}{1-x^2},x\Big{)}\cdot\frac{1}{1-x^2}\gg\frac{(x-1)(x+1)(x^6+x^5-3x^4+x^3+3x^2-1)}
{1-5x^2-2x^3+10x^4+2x^5-9x^6+2x^7+5x^8-2x^9-x^{10}}.
\end{eqnarray*}  And indeed, minding the values given by (\ref{seka}), MAPLE confirms that the first disrepancy occurs only for $n=13$. Namely, $286>281$, and the five missing words are precisely those given by Example \ref{ex2} in the sets $``d"$ through $``g"$. For $n=14$ there are no words in the sets beyond $``c"$, so the above is in fact the equality $722=722$. The Galois group of the splitting field of the polynomial in the denominator is equal to $S_{10}$. This polynomial has the unique root $\theta$ of the smallest absolute value, it is real  and positive: $\theta=0.5394737936_{+}$, $\theta^{-1}=1.853658161_{+}$. So, this single observation gives the lower bound
\begin{eqnarray*}
t(n)>C(1.853658161)^n
\end{eqnarray*}
for a certain $C>0$. In fact, our main Theorem implies that for the function $t(n)$ we have the sharper bound
\begin{eqnarray*}
t(n)>C(1.971480194)^n,
\end{eqnarray*}
where $\phi^{-1}=1.971480194_{+}$, $\phi$ being the smallest (in absolute value) root of $x^7-20x^5+12x^4-8x^3-12x^2+4=0$, the factor of the discriminant of the cubic polynomial in the Theorem.
\end{Example}
To prove the formula for $A$, and hence for $\widehat{Q}$, we need to express $\mathfrak{D}$ and $\mathfrak{F}$ in terms of $A$. In order to accomplish this, we introduce the notion of a \emph{primitive} $\mathbf{a}$-\emph{word}. This, by definition, is the word which belongs to the subset which defines the function $``a"$ (an $\mathbf{a}-$\emph{word}), but which cannot be written as
\begin{eqnarray}
A_{1}U^{\alpha_{1}}A_{2}U^{\alpha_{2}}\cdots U^{\alpha_{s-1}}A_{s},\quad s\geq 2,\quad \alpha_{i}\geq 1,
\label{a-word}
\end{eqnarray}
where each $A_{i}$ is an $\mathbf{a}-$word.\\

Thus, let us continue our classification given by $a)$ through $\mathfrak{f})$, and introduce
\begin{itemize}
\item[$\mathfrak{a}$)]$\mathfrak{a}(n,m)$ --- number of primitive $\mathbf{a}$-words, which have $n$ copies of $U$ and $m$ copies of $S$.
\end{itemize}
\begin{Example}In the Example \ref{ex2} above, six of the $\mathbf{a}-$words are primitive, only $SU^3SU^3SU^3S$ is not.
\end{Example}
\begin{prop}We have the recurrence 
\begin{eqnarray*}
& &a(3n,2m)=\mathfrak{a}(3n,2m)
+\sum\limits_{k=1}^{\infty}\sum\limits_{3n_{1}+3n_{2}=3n-3k
\atop 2m_{1}+2m_{2}=2m}\mathfrak{a}(3n_{1},2m_{1}) \mathfrak{a}(3n_{2},2m_{2})\\
&+&
\sum\limits_{k=1}^{\infty}\binom{3k-1}{1}\sum\limits_{3n_{1}+3n_{2}+3n_{3}=3n-3k
\atop 2m_{1}+2m_{2}+2m_{3}=2m}\mathfrak{a}(3n_{1},2m_{1}) \mathfrak{a}(3n_{2},2m_{2})\mathfrak{a}(3n_{3},2m_{3})\\
&+&
\sum\limits_{k=1}^{\infty}\binom{3k-1}{2}\sum\limits_{3n_{1}+3n_{2}+3n_{3}+3n_{4}=3n-3k
\atop 2m_{1}+2m_{2}+2m_{3}+2m_{4}=2m}\mathfrak{a}(3n_{1},2m_{1}) \mathfrak{a}(3n_{2},2m_{2}) \mathfrak{a}(3n_{3},2m_{3})\mathfrak{a}(3n_{4},2m_{4})\\
&+&\cdots.
\end{eqnarray*}
\label{prop4} 
\end{prop}
\begin{proof}We note that each $\mathbf{a}-$word is either primitive, or can be written in the form (\ref{a-word}), where each $A_{i}$ is a primitive $\mathbf{a}-$word. This claim follows by induction. Now we are left to count, which gives the above formula. Here $\binom{3k-1}{s}$ stands for a number of ways the number $3k$ can be written as a sum of $s+1$ positive integers. 
\end{proof}
Let
\begin{eqnarray*}
W(x,y)=\sum\limits_{n,m\geq 0}\mathfrak{a}(3n,2m)x^{3n}y^{2m}.
\end{eqnarray*}
The identity in Proposition \ref{prop4} can be written as
\begin{eqnarray*}
A=W+\sum\limits_{k=1}^{\infty}\sum\limits_{s\geq 0}\binom{3k-1}{s}x^{3k}W^{s+2}=
W+\sum\limits_{k=1}^{\infty}x^{3k}W^{2}(1+W)^{3k-1}=
W+\frac{x^3W^2(1+W)^2}{1-x^3(1+W)^3}.
\end{eqnarray*}
Analogously we derive recurrences for $\mathfrak{d}$ and $\mathfrak{f}$ in terms of $\mathfrak{a}$, which lead to the identities  
\begin{eqnarray*}
\mathfrak{D}=\frac{x^3W^2}{1-x^3(1+W)^3},\quad \mathfrak{F}=
\frac{x^3W^2(1+W)}{1-x^3(1+W)^3}.
\end{eqnarray*}
Plugging all these three identities for $A$, $\mathfrak{D}$ and $\mathfrak{F}$ into (\ref{a-df}), we readily obtain that $W$ is an algebraic function:\begin{small}
\begin{eqnarray*}
W=\frac{y^2x^3(1-x^3)}{(1-x^3)^2-y^2x^3(x^3+2)}-
\frac{W^2x^3(1+W)^2}{1-x^3(1+W)^3}
+\frac{W^2x^3y^2(5x^3+1+3x^3W)
}{[1-x^3(1+W)^3][(1-x^3)^2-y^2x^3(x^3+2)]}.
\end{eqnarray*}
\end{small}
MAPLE simplifies this to a very elegant form
\begin{eqnarray}
W=x^3(W+1)^2(W+y^2).
\label{ww}
\end{eqnarray}
So, $W$ is a third degree algebraic function over $\mathbb{Q}(x,y)$, and thus so is $Q$. The exact form of the cubic equation can be easily calculated with MAPLE, but we rather concentrate on $T(x)$, since, first, the equation for $Q$ with the help of method in Section \ref{sec-alt} can be easily calculated by hand, and second, the cubic equation for $W$ is very convenient to calculate fast the Taylor coefficients of $T$ recurrently.
Let
\begin{eqnarray*}
Z(x)=W\Big{(}\frac{x}{1-x^2},x\Big{)}.
\end{eqnarray*}
Then the equation for $Z$ reads as
\begin{eqnarray}
Z=\frac{x^{3}(Z+1)^2(Z+x^2)}{(1-x^2)^3}.
\label{z}
\end{eqnarray}
We are left to verify the cubic equation given in the formulation of the Theorem. Plugging known values into (\ref{tarpp}) and using Proposition \ref{prop2}, we obtain
\begin{eqnarray}
T(x)=\frac{(1-x^2)^2(1+Z)}{1-3x^2-x^3+3x^4-x^6-3x^3Z-3x^3Z^2-x^3Z^3}.
\label{t}
\end{eqnarray}
Now, plug this value of $T(x)$ into the cubic equation given by the Theorem. Then factor the numerator. MAPLE confirms that one of the two multipliers is indeed $Z(1-x^2)^3-x^3(Z+1)^2(Z+x^2)$, so the equation for $T(x)$ is verified. In fact, this equation was discovered by Robert Israel by finding a $3$rd degree algabraic function whose first $42$ Taylor coefficients coincide with $t(n)$, $0\leq n\leq 41$. Our theoretical result thus double-checks this fact.\\

The equalities (\ref{z}) and (\ref{t}) give a polynomial-time method to calculate coefficients $t(n)$. Indeed, let us start from $Z_{1}(x)=x^5$ (the first primitive $\mathbf{a}-$word is $SU^3S$), and let us define polynomials $Z_{N}(x)\in\mathbb{Z}[x]$ recurrently by 
\begin{eqnarray*}
Z_{N+1}=\frac{x^{3}(Z_{N}+1)^2(Z_{N}+x^2)}{(1-x^2)^3}\text{ (mod }x^{3N+6}).
\end{eqnarray*}
Thus, $Z_{N}$ is of degree $3N+2$, $Z_{N+1}\equiv Z_{N}\text{ (mod }x^{3N+3})$. After reaching enough terms, plug this into (\ref{t}). This agrees perfectly with (\ref{seka}), which was calculated by a direct count. Robert Israel calculated $2000$ terms of the sequence $t(n)$. 

\section{The short proof using PDA}
\label{sec-alt}
The following alternative proof was proposed by the referee. We reproduce it almost \emph{verbatim}, since the ideas are clear and self-explanatory. The \emph{PDA} stands for a \emph{pushdown automaton}. Both proofs give exactly the same algebraic equation for the function $Q$.\\

The group $\Gamma=\sf{PSL}_{2}(\mathbb{Z})$ is \emph{virtually free}, so in the framework of geometric and combinatoric group theory it is known that the \emph{word problem} (the set of all words in generators and inverses that are equal 1) is an unambiguous context-free language. However, for our purposes in \cite{alkauskas} we need the regular language of all \emph{positive words} $\{U,S\}^{*}$ to obtain the language of words counted by the generating functions $T$ and $Q$. It follows immediatelly from the Chomsky-Sch\"{u}tzenberger enumeration theorem that the functions $T$ and $Q$ are algebraic.\\

To obtain the explicit formulae, we can explicitly construct a PDA accepting the language, then follow standard methods to obtain generating functions from the PDA.\\

The language is very simple: a word equals $1$ if some sequence of applications of the rules $S^2\rightarrow 1$, $U^3\rightarrow 1$ to factors of
the input word reduces it to the empty word. This can be done using
a PDA which has just one state, and a pushdown stack which
uses the alphabet $\{0,1,2,3\}$ where $0$ is the bottom-of-stack marker, $1$ means a single $U$, $2$ means $U^2$, and $3$ means $S$. So that the model of a PDA can be used that accepts
on empty stack. Let us introduce a symbol $\dollar$, and consider the language $L\dollar=\{w\dollar\,|\,w\in\{U,S\}^{*},w=_{\Gamma}1\}$.\\

Here are the transitions:
\begin{itemize}
\item[1)]$U,0\rightarrow 10$
\item[2)]$U,1\rightarrow 2$
\item[3)]$U,2\rightarrow \epsilon$
\item[4)]$U,3\rightarrow 13$
\item[5)]$S,0\rightarrow 30$
\item[6)]$S,1\rightarrow 31$
\item[7)]$S,2\rightarrow 32$
\item[8)]$S,3\rightarrow \epsilon$
\item[9)]$\dollar,0\rightarrow \epsilon$
\end{itemize}

This is indeed self-explanatory. For example, take the item $4)$. This means $U$ is the next letter, and $3$ is on the top of the stack (that is, $S$). Since $US$ does not reduce, replace $3$ with $13$. On the other hand, consider the item $3)$. It stands for a move that now we get a factor $U^{3}$, which is removed.\\

The idea is that one reads a word in $U,S$ and puts it into \emph{normal form} on the fly, reducing $U^3$, $S^2$ whenever they appear as one moves right. The normal form of the prefix of the input word is written (in code) on the stack. Accept if at the end the stack is empty (the normal form is $1$).\\

Now we follow \cite{hopcroft} to convert the deterministic PDA into an unambiguous context free grammar. Let $N_{i}$ stand for the nonterminal $N_{q,i,q}$ (since there exists only one state). Start symbol is $N_{0}$. We get
\begin{itemize}
\item[1)]$N_{0}\rightarrow UN_{1}N_{0}$
\item[2)]$N_{1}\rightarrow UN_{2}$
\item[3)]$N_{2}\rightarrow U$
\item[4)]$N_{3}\rightarrow UN_{1}N_{3}$
\item[5)]$N_{0}\rightarrow SN_{3}N_{0}$
\item[6)]$N_{1}\rightarrow SN_{3}N_{1}$
\item[7)]$N_{2}\rightarrow SN_{3}N_{2}$
\item[8)]$N_{3}\rightarrow S$
\item[9)]$N_{0}\rightarrow \dollar$
\end{itemize}
Next, following the method of Chomsky-Sch\"{u}tzenberger 
we obtain the following system of equations for the generating function directly from the grammar:
\begin{eqnarray}
\left\{\begin{array}{l}
f_{0}=xf_{1}f_{0}+yf_{3}f_{0}+z,\\
f_{1}=xf_{2}+yf_{3}f_{1},\\
f_{2}=x+yf_{3}f_{2},\\
f_{3}=xf_{1}f_{3}+y.
\end{array}
\right.
\label{sys-in}
\end{eqnarray}
That is, we replace $U$ and $S$ with $x$ and $y$, respectively, and add the corresponding terms. If $(f_{0},f_{1},f_{1},f_{3})$ is the solution, then $Q(x,y)=\frac{f_{0}}{z}$. Note that the variable $z$ appears only in the extression for $f_{0}$ only as a linear factor $z$.\\

Solving easily by hand the last three equations of (\ref{sys-in}) gives the cubic equation for $f_{3}=K$:
\begin{eqnarray}
y^2K^3-(2y+y^3)K^2+(1+2y^2-x^3)K-y=0.
\label{cub}
\end{eqnarray}
Plugging everything into the first equation of (\ref{sys-in}), we get 
\begin{eqnarray}
Q(x,y)=\frac{K}{y(1-K^2)}.
\label{q-k}
\end{eqnarray}
Let $K_{1}$, $K_{2}$ and $K_{3}$ be three distinct roots of (\ref{cub}), and let $\sigma_{1}=K_{1}+K_{2}+K_{3}$, $\sigma_{2}=K_{1}K_{2}+K_{1}K_{3}+K_{2}K_{3}$, $\sigma_{3}=K_{1}K_{2}K_{3}$ be the standard symmetric polynomials. Then
\begin{eqnarray*}
\sigma_{1}=\frac{2+y^2}{y},\quad \sigma_{2}=\frac{1+2y^2-x^3}{y^2},\quad \sigma_{3}=\frac{1}{y}.
\end{eqnarray*}
 Let $Q_{i}$, $1\leq i\leq 3$, are obtained from (\ref{q-k}) by plugging $K_{i}$ instead of $K$. We get
\begin{eqnarray*}
\frac{1}{Q_{1}}+\frac{1}{Q_{2}}+\frac{1}{Q_{3}}&=&y\sum \frac{1}{K_{i}}-y\sum K_{i}=\frac{y\sigma_{2}}{\sigma_{3}}-y\sigma_{1}=y^2-x^3-1,\\
\frac{1}{Q_{1}Q_{2}}+\frac{1}{Q_{2}Q_{3}}+\frac{1}{Q_{1}Q_{3}}&=&\frac{y^2\sigma_{1}}{\sigma_{3}}+y^2\sigma_{2}-
y^2\frac{(\sigma_{1}\sigma_{2}-3\sigma_{3})}{\sigma_{3}}
=x^3y^2-y^4+x^3+2y^2-1,\\
\frac{1}{Q_{1}Q_{2}Q_{3}}&=&\frac{y^3}{\sigma_{3}}-y^3\sigma_{3}
-y^{3}\frac{\sigma_{1}^2-2\sigma_{2}}{\sigma_{3}}+y^3\frac{\sigma_{2}^{2}-2\sigma_{1}\sigma_{3}}{\sigma_{3}}\\
&=&x^6-y^6-6y^2x^3+3y^4-2x^3-3y^2+1.
\end{eqnarray*}
We thus get the cubic equation for $\frac{1}{Q}$, and the reciprocal of it is the equation for $Q$, exactly as formulated in the Theorem.

\subsection*{Acknowledgements} I would like to thank Wadim Zudilin, Murray Elder and Franz Lehner for sending me references and pointing out that these results are a small part of the much richer field in the theory of groups and graphs. I sincerely thank Roland Bacher, and also Robert Israel and activists of OEIS for helping with MAPLE and the sequence $t(n)$ itself. I thank Audrius Alkauskas for a picture. I especially thank the anonymous referee: all the text in Section \ref{sec-alt}, up to the system (\ref{sys-in}), is almost verbatim taken from the referee's report.


\begin{thebibliography}{99}

\bibitem{alkauskas} {\sc G. Alkauskas}, The Minkowski $?(x)$ function, a class of singular measures, quasi-modular and mean-modular forms. \url{http://arxiv.org/abs/1209.4588}.


\bibitem{bartholdi}{\sc L. Bartholdi}, Counting paths in graphs, {\it Enseign. Math. (2)} {\bf 45} (1-2) (1999), 83--131.

\bibitem{finch}{\sc S.R. Finch}, {\it Mathematical constants}, Encyclopedia of Mathematics and its Applications, 94. Cambridge University Press, Cambridge (2003).

\bibitem{hopcroft}{\sc J.E. Hopcroft, J.D. Ullman}, {\it Introduction to automata theory, languages, and computation}, Addison-Wesley Series in Computer Science. Addison-Wesley Publishing Co., Reading, Mass. (1979).

\bibitem{kuksov2}{\sc D. G. Kouksov}, On rationality of the cogrowth series, {\it Proc. Amer. Math. Soc.} {\bf 126} (10) (1998), 2845--2847.

\bibitem{kuksov}{\sc D. Kuksov}, Cogrowth series of free products of finite and free groups, {\it Glasgow Math. J.} {\bf 41} (1) (1999), 19--31.

 
\bibitem{lehner}{\sc F. Lehner}, On the computation of spectra in free probability, {\it J. Funct. Anal.} {\bf 183} (2) (2001), 451--471. 
 
\bibitem{thesis}{\sc J. C. McLaughlin}, ``Random walks and convolution operators on free
products", Doctoral dissertation, New York University, 1986.

\bibitem{quenell}{\sc G. Quenell}, Combinatorics of free product graphs, {\it Contemp. Math.} {\bf 173} (1994), 257--281.


\bibitem{voiculescu}{\sc D. Voiculescu}, Addition of certain noncommuting random variables, {\it J. Funct. Anal.} {\bf 66} (3) (1986), 323--346.

\bibitem{woess1}{\sc W. Woess}, Nearest neighbour random walks on free products of discrete groups. {\it Boll. Un. Mat. Ital. B} (6) {\bf 5} (3) (1986), 961--982.

\bibitem{woess2}{\sc W. Woess}, {\it Random walks on infinite graphs and groups}. Cambridge Tracts in Mathematics, 138. Cambridge University Press, Cambridge, 2000. 

\bibitem{oeis}The Online Encyclopedia of Integer Sequences, Sequence A265434.



\end{thebibliography}
\end{document}